\documentclass[11pt,reqno]{amsart}
\usepackage{graphicx}
\usepackage{verbatim}
\usepackage{textcomp}
\usepackage{amssymb}
\usepackage{cite}
\usepackage{amsmath}
\usepackage{latexsym}
\usepackage{amscd}
\usepackage{amsthm}
\usepackage{mathrsfs}
\usepackage{xypic}
\usepackage{bm}
\usepackage{url}
\usepackage{hyperref}

\newtheorem{thm}{Theorem}
\newtheorem{corr}[thm]{Corollary}
\newtheorem{lem}[thm]{Lemma}
\newtheorem{prop}[thm]{Proposition}
\newtheorem{conj}[thm]{Conjecture}

\theoremstyle{definition}
\newtheorem{defn}{Definition}[section]

\newtheorem{rem}[]{Remark}
\def\R{\mathbb R}

\def\SS{\mathbb S}
\def\cal{\mathcal}
\def\f{\frac}
\def\td{\tilde}
\def\b{\bar}
\def\bn{\bar{\nu}}
\def\la{\langle}
\def\ra{\rangle}
\def\pt{\partial}
\def\e{\varepsilon}

\begin{document}
\title{Stability of capillary hypersurfaces in a Euclidean ball}
\author{Haizhong Li}
\address{Department of mathematical sciences, Tsinghua University, 100084, Beijing, P. R. China}
\email{hli@math.tsinghua.edu.cn}
\author{Changwei Xiong}
\address{Department of mathematical sciences, Tsinghua University, 100084, Beijing, P. R. China}
\email{xiongcw10@mails.tsinghua.edu.cn}
\thanks{The research of the authors was supported by NSFC No. 11271214.}

\maketitle

\begin{abstract}
We study the stability of capillary hypersurfaces in a unit Euclidean ball. It is proved that if the mass center of the generalized body enclosed by the immersed capillary
hypersurface and the wetted part of the sphere is located at the origin, then the hypersurface is unstable. An immediate result is that all known examples except the totally
geodesic ones and spherical caps are unstable.
\end{abstract}

\section{Introduction}

Capillarity is an important physical phenomena, which occurs when two different materials contact and do not mix. Given a container $B$ with an incompressible liquid drop $T$ in it, the interface of the liquid and the air is a capillary surface $M$. In absence of gravity, the interface $M$ is of constant mean curvature and the contact angle of $M$ to the boundary $\pt B$ is constant. One should compare this setting with soap bubble (resp. soap film), where the surface has no boundary (resp. fixed boundary) and constant mean curvature.

The literature for the study of capillarity is extensive and we refer to the book of Finn \cite{F1}, where the treatment of the theory is mainly in the nonparametric case and in the more general situation of presence of gravity. Also we mention \cite{F2} for a more recent survey about this topic.

In this paper we are concerned with the special case that the container $B$ is a unit Euclidean ball and no gravity is involved. We study the (weak) stability for capillary hypersurfaces. This problem has been discussed by Ros and Souam \cite{RS}, where they dealt with surface case and obtained some topological and geometrical restrictions. For the hypersurface case with free boundary (the contact angle is $\pi/2$), Ros and Vergasta also proved some interesting results in \cite{RV}.

Applying the same argument as in the proof of Proposition 1.1 in \cite{RS}, we know that the hyperplanes and the spherical caps in a unit Euclidean ball are capillarily stable. Recently, Marinov \cite{M} proved that, for surface case, all other known examples are unstable. We generalize Marinov's result to the hypersurface case. In fact we prove the following theorem.

\begin{thm}\label{thm1}
Let $x:M^n\rightarrow \R^{n+1}$ be an immersed capillary hypersurface in a unit Euclidean ball $B^{n+1}$ and $\Omega$ the wetted part of the boundary of the ball. Denote by $T$ the generalized body enclosed by $x(M)$ and $\Omega$. If the mass center of $T$ is at the origin, the capillary hypersurface $M$ is unstable.
\end{thm}

Here since we assume $M$ is immersed, $x(M)$ may have self-intersections. Thus we need to consider the generalized body $T$. When $M$ is embedded, $T$ is understood in the common sense. See Remark \ref{rem1} below for more explanation.

For $n=2$, our Theorem \ref{thm1} reduces to Marinov's result in \cite{M}. We also note that his argument relies on the conformal coordinates, which can not be generalized to the higher dimensional case.

Applying Theorem \ref{thm1} to Delaunay hypersurfaces, we get the following Corollary \ref{cor1}. Recall that Delaunay hypersurfaces are the ones of revolution with constant mean curvature. By Proposition 4.3 in \cite{HMRR}, Delaunay hypersurfaces are classified as: an unduloid, cylinder, nodoid, sphere, catenoid, or a hyperplane. To guarantee the portion of a Delaunay hypersurface in a Euclidean ball is also capillary, it shall have some symmetry. See Section \ref{sec2} below for more details. And in that case, we call it Delaunay capillary hypersurface. From Theorem \ref{thm1} we have
\begin{corr}\label{cor1}
The only stable Delaunay capillary hypersurface $M^n$ in a unit Euclidean ball $B^{n+1}$ is a totally geodesic hyperplane or a spherical cap.
\end{corr}

Our approach for proving Theorem \ref{thm1} is as follows. In higher dimensional case, we find that we can construct a conformal killing vector field $Y[\xi]$ for any fixed $\xi\in\SS^n$ from the natural conformal transformation family on $B^{n+1}$. Using the normal part $\la Y[\xi],N\ra$ as the test function we can define a symmetric quadratic form $Q(\xi_1,\xi_2)$. By summing $Q$ over $(n+1)$ coordinate directions we find $Q$ has at least one negative eigenvalue. This summation technique can be compared with J.~Simons' work \cite{S}. At last, under the hypothesis of Theorem \ref{thm1} we can derive the instability of the hypersurface. Our argument indicates that this conformal field is very important and we can use it to conclude that the mass center of minimal submanifolds with free boundary in a unit Euclidean ball is at the origin (See Proposition \ref{prop1}). We refer the readers to \cite{FS,FS1,FS2} for the very recent work on the minimal submanifolds with free boundary.

At last, as an application of our argument, we give a new proof of the classical result due to Barbosa and do Carmo \cite{BdC}, which states that the only closed stable immersed hypersurface of constant mean curvature in $\R^{n+1}$ is the round sphere.

An outline of this paper is as follows. In Section \ref{sec2} after fixing some notations and definitions, we prove the stability of hyperplanes and spherical caps. Then we construct the crucial conformal vector field. We also review some known facts about the Delaunay hypersurfaces. In Section \ref{sec3} we give the proof of Theorem \ref{thm1}. In last section, we discuss some applications of our method.

\section{Preliminaries}\label{sec2}

\subsection{Notations and definitions}

Let $x:M^n\rightarrow \R^{n+1}$ be an orientable immersed hypersurface in the unit Euclidean ball $B^{n+1}\subset \R^{n+1}$ with $x(int\ M)\subset B^{n+1}$ and $x(\pt M)\subset \pt B^{n+1}$. Suppose $\Omega\subset \pt B^{n+1}$ such that $\pt \Omega=x(\pt M)$. And denote by $T\subset B^{n+1}$ the part of ball satisfying $\pt T=x(M)\cup \Omega$.

\begin{rem}\label{rem1}
If $x(M)$ has self-intersections, $T$ may be viewed as the finite union of some domains $T_i, i=1,\cdots,m$, i.e. $T=\cup_{i=1}^m T_i$. Here $T_i$ may intersect with each other. If there are not one choices for $\{T_i\}_{i=1}^m$, choose one and fix it. In the proof we will see that only the property $\pt T=x(M)\cup \Omega$ is needed. And if there is no confusion, we will write $M$ (resp. $\pt M$) for $x(M)$ (resp. $x(\pt M)$) for simplicity.
\end{rem}

Let $N$ be the unit normal of $M$ pointing inwards to $T$ and $\b{N}$ the unit outward normal of $\pt B^{n+1}$. Denote by $\nu$ and $\bn$ the conormals of $\pt M$ in $M$ and $\Omega$, respectively. Let $D$ (resp. $\nabla$) be the connection of $\R^{n+1}$ (resp. $M$). Then the second fundamental form of $M$ in $\R^{n+1}$ is given by $\sigma(X_1,X_2)=\la D_{X_1}X_2,N\ra$ for $\forall\, X_1,X_2\in TM$. When taking an orthonormal basis $\{e_i\}_{i=1}^n$ on $TM$, we also denote by $h_{ij}$ the components $\sigma(e_i,e_j)$. So the mean curvature $H$ of $M$ is $H=\f{1}{n}\sum\limits_{i=1}^n h_{ii}$. And the second fundamental form of $\pt B$ in $\R^{n+1}$ is given by $\Pi(Y_1,Y_2)=\la D_{Y_1}Y_2,-\b{N} \ra$ for $\forall\, Y_1,Y_2\in T(\pt B)$. At last let $\theta\in (0,\pi)$ be the angle between $\nu$ and $\bn$. See Figure \ref{pic} for an illustration.

\begin{figure}
\includegraphics[scale=0.7]{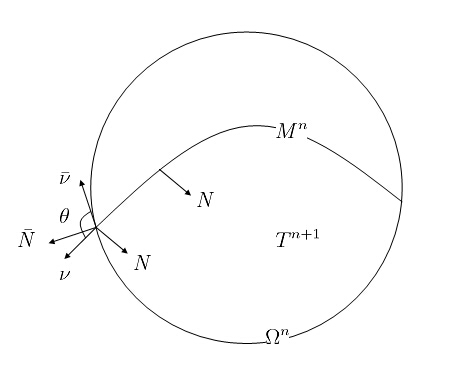}
\caption{A typical illustration\label{pic}}
\end{figure}

Following \cite{RS}, we discuss the variation of $M$.
\begin{defn}
An admissible variation of $x:M^n\rightarrow \R^{n+1}$ is a differentiable map $X: (-\e,\e)\times M\rightarrow \R^{n+1}$ so that $X_t:M^n\rightarrow \R^{n+1}, t\in(-\e,\e)$ given by $X_t(p)=X(t,p), p\in M$ is an immersion satisfying $X_t(int\ M)\subset int\ B$ and $X_t(\pt M)\subset \pt B$ for all $t$, and $X_0=x$.
\end{defn}

Now for given $\theta \in (0,\pi)$, we define a energy functional
\begin{equation}
E(t)=|M(t)|-\cos \theta |\Omega(t)|,
\end{equation}
where $|\cdot|$ denotes the area function. And the volume functional can be defined as
\begin{equation*}
V(t)=\int_{[0,t]\times M} X^*dv,
\end{equation*}
where $dv$ is the standard volume element of $\R^{n+1}$.

Under these constraints, we have
\begin{defn}
An immersed hypersurface $x:M^n\rightarrow \R^{n+1}$ is called capillary if $E'(0)=0$ for any admissible volume-preserving variation of $x$.
\end{defn}
Note that we have the following formulas,
\begin{align}
E'(0)&=-n\int_M Hfda+\int_{\pt M} \la Y,\nu-\cos \theta \bn\ra ds,\\
V'(0)&=-\int_M f da,
\end{align}
where $Y$ is the variational vector field $Y(p)=\f{\pt X}{\pt t}(p)|_{t=0}$, $f$ its normal component $f=\la Y,N\ra$, and $da$ and $ds$ are the corresponding area elements.

From these formulas we see that $M$ is capillary if and only if it has constant mean curvature and makes constant contact angle $\theta$ with $\pt B$. Furthermore, one can compute the second derivative at $t=0$ with respect to an admissible volume-preserving variation to get (see e.g. the appendix of \cite{RS})
\begin{equation}
E''(0)=-\int_M \left(\Delta f+(|\sigma|^2+Ric(N))f\right)fda+\int_{\pt M} (\f{\pt f}{\pt \nu}-qf)fds,
\end{equation}
where $f\in \cal{F}:=\{f\in H^1(M),\int_M fda=0\}$, $Ric(N)$ is the Ricci curvature of the ambient space and
\begin{equation}
q=\f{1}{\sin \theta}\Pi(\bn,\bn)+\cot \theta \sigma(\nu,\nu).
\end{equation}
In our setting, $Ric(N)=0$ and $\Pi(\bn,\bn)=1$.

\begin{defn}
A capillary hypersurface $M$ is called (weakly) stable if $E''(0)\geq 0$ for all $f\in \cal{F}$.
\end{defn}

In the sequel we will denote by $\pt^2 E(\phi)$ the quantity $E''(0)$ for a given function $\phi$.

\subsection{Stable examples of capillary hypersurfaces}

First we prove the stability of totally geodesic capillary hypersurfaces and spherical caps. The proof is similar to that of Proposition 1.1 in \cite{RS}. We include it for completeness.
\begin{prop}
Let $B^{n+1}\subset \R^{n+1}$ be a unit Euclidean ball. Then totally geodesic capillary hypersurfaces and spherical caps are stable.
\end{prop}
\begin{proof}
First assume $M$ is a totally geodesic capillary hypersurface, i.e., an $n$-dimensional ball $B^n(R)$ with radius $R$ in $B^{n+1}$. Then the contact angle $\theta$ satisfies $\sin \theta =R$. By the definition of stability, we have to prove
\begin{equation}\label{eq2}
\int_M|\nabla f|^2 da\geq \f{1}{R}\int_{\pt M}f^2 ds,\, \text{for}\,\,\forall f\in \cal{F}.
\end{equation}
Consider now the $(n+1)$-dimensional ball $B'$ of radius $R$ having $M$ as an equatorial totally geodesic hypersurface. Then by \cite{BS}, $M$ is area minimizing for partitioning problem in $B'$. Thus $M$ is stable in $B'$, which is equivalent to the inequality \eqref{eq2}.

Next assume $M$ is a spherical cap in $B^{n+1}$ with $R$ the radius of the sphere containing $M$ and $\theta$ the contact angle. Consider the $n$-dimensional hyperplane $P$ containing $\pt M$. Then $M$ is a capillary hypersurface in a halfspace with a contact angle $\theta'$. By \cite{GMT}, $M$ is stable in the halfspace, which means
\begin{equation}\label{eq3}
\int_M(|\nabla f|^2-\f{n}{R^2}f^2) da\geq \f{\cot \theta'}{R}\int_{\pt M}f^2 ds,\, \text{for}\,\,\forall f\in \cal{F}.
\end{equation}
Elementary calculation leads to
\begin{equation}\label{eq4}
\f{1}{\sin \theta}+\f{\cot \theta}{R}=\f{\cot \theta'}{R}.
\end{equation}
Now \eqref{eq3} and \eqref{eq4} together yield the stability of $M$ in $B^{n+1}$.

\end{proof}

\subsection{Conformal transformations on the Euclidean ball}

Now we construct a conformal vector field. Fix a vector $a\in B^{n+1}$. Then (see e.g. Section $3.8$ in \cite{SY})
\begin{equation}
\varphi_a(x)=\f{(1-|a|^2)x-(1-2\langle a,x\rangle+|x|^2)a}{1-2\langle a,x\rangle+|a|^2|x|^2}
\end{equation}
defines a map from $B^{n+1}$ to $B^{n+1}$ and from $\SS^n$ to $\SS^n$, since we have
\begin{equation*}
1-|\varphi_a(x)|^2=\f{(1-|a|^2)(1-|x|^2)}{1-2\la a,x\ra+|a|^2|x|^2}.
\end{equation*}

Moreover $\varphi_a$ is conformal. In fact, by a direct calculation we can check that
\begin{equation*}
|d\varphi_a|^2=\left(\f{1-|a|^2}{1-2\la a,x\ra+|a|^2|x|^2}\right)^2|dx|^2.
\end{equation*}

Note that $\varphi_a(a)=0$, $\varphi_a(0)=-a$, $\varphi_a$ fixes two points $\pm \f{a}{|a|}$ and $\varphi_0$ is an identity.

Next fix $\xi\in \SS^n$. Let $a=t\xi$ with $-1<t<1$. Then
\begin{equation}
f_t(x)=\varphi_{t\xi}(x)=\f{(1-t^2)x-(1-2t\langle \xi,x\rangle+|x|^2)t\xi}{1-2t\langle \xi,x\rangle+t^2|x|^2}
\end{equation}
is a family of conformal transformations with parameter $t$.

Thus $f_t$ determines a conformal vector field $Y[\xi]$ as follows.
\begin{equation}
Y[\xi]=\f{d}{dt}\bigg|_{t=0}f_t(x)=-(1+|x|^2)\xi+2\langle \xi,x\rangle x.
\end{equation}

Note that $Y[\xi]$ is tangential along the sphere $\SS^n$, since for $\forall x\in \SS^n$,
\begin{equation*}
\la Y[\xi],x\ra=-(1+|x|^2)\la \xi,x\ra+2\la \xi,x\ra |x|^2=0.
\end{equation*}

\subsection{Delaunay hypersurfaces in Euclidean space}

In this subsection, following \cite{HMRR} we review some facts about Delaunay hypersurfaces, which is rotational and of constant mean curvature $H$. These hypersurfaces are the models we are concerned with in Theorem \ref{thm1}.

Let $M^n\subset \R^{n+1}$ be a hypersurface which is invariant under the action of the orthogonal group $O(n)$ fixing the $x^1$-axis. Assume $M$ is generated by a curve $\Gamma$ contained in the $x^1x^2$-plane. Then it suffices to determine the curve $\Gamma$.

Parametrize the curve $\Gamma=(x^1,x^2)$ by arc-length $s$. Denote by $\alpha$ the angle between the tangent to $\Gamma$ and the positive $x^1$-direction and choose the normal vector $N=(\sin \alpha,-\cos \alpha)$. Then $(x^1,x^2;\alpha)$ satisfies the following system of ordinary differential equations
\begin{equation*}
\begin{cases}
(x^1)'=\cos \alpha,\\
(x^2)'=\sin \alpha,\\
\alpha'=-nH+(n-1)\f{\cos \alpha}{x^2}.
\end{cases}
\end{equation*}
The first integral of this system is given by
\begin{equation*}
(x^2)^{n-1}\cos \alpha-H(x^2)^n=F,
\end{equation*}
where the constant $F$ is called the force of the curve $\Gamma$ and it together with $H$ will determine the curve as follows. (See Proposition 4.3 in \cite{HMRR})

\begin{prop}
The curve $\Gamma$ and the hypersurface $M$ generated by $\Gamma$ have the following several possible types.
\begin{enumerate}
  \item If $FH>0$ then $\Gamma$ is a periodic graph over the $x^1$-axis. It generates a periodic embedded unduloid, or a cylinder.
  \item If $FH<0$ then $\Gamma$ is a locally convex curve and $M$ is a nodoid, which has self-intersections.
  \item If $F=0$ and $H\neq 0$ then $M$ is a sphere.
  \item If $H=0$ and $F\neq 0$ we obtain a catenary which generates an embedded catenoid $M$ with $F>0$ if the normal points down and $F<0$ if the normal points up.
  \item If $H=0$ and $F=0$ then $\Gamma$ is a straight line orthogonal to the $x^1$-axis which generates a hyperplane.
  \item If $M$ touches the $x^1$-axis, then it must be a sphere or a hyperplane.
  \item The curve $\Gamma$ is determined, up to translation along the $x^1$-axis, by the pair $(H,F)$.
\end{enumerate}
\end{prop}

From this proposition, it is easy to see if $M^n$ is the portion of an unduloid, cylinder, nodoid or a catenoid in a unit Euclidean ball $B^{n+1}$ with revolution axis $x^1$ and moreover $M$ is symmetric with respect to the hyperplane $\{x^1=0\}$, then $M$ is a capillary hypersurface in $B^{n+1}$. In that case we call them Delaunay capillary hypersurfaces in $B^{n+1}$. Furthermore the generalized body $T$ enclosed by $M$ and the wetted part of the sphere has the mass center at the origin. So Theorem \ref{thm1} is applicable.

\section{Instability of capillary hypersurfaces}\label{sec3}

With the preparations above, we can define a ``test function''
\begin{equation}
\phi[\xi]=\langle Y[\xi],N\rangle=\langle -(1+|x|^2)\xi+2\langle \xi,x\rangle x, N\rangle.
\end{equation}
We mention that we will also use the following expression of $\phi[\xi]$
\begin{equation}\label{eq7}
\phi[\xi]=\langle \xi, -(1+|x|^2)N+2\langle x,N\rangle x\rangle.
\end{equation}

Recall the second variational formula
\begin{equation}
\pt^2 E(\phi)=-\int_M L\phi \cdot \phi da+\int_{\pt M} (\phi_\nu-q\phi)\phi ds,
\end{equation}
where $L=\Delta+|\sigma|^2$ and $q=\f{1}{\sin \theta}+\cot \theta \sigma(\nu,\nu)$.

Now we can prove the following lemmas.
\begin{lem}\label{lem3}
$\nu$ is a principal direction for $\sigma$ along $\pt M$. In particular, $D_{\nu} N=-\sigma(\nu,\nu)\nu$.
\end{lem}
\begin{proof}
It suffices to prove that, for $\forall X\in T_p (\pt M)$, $\sigma(\nu, X)=0$. In fact, we have
\begin{align*}
\sigma(\nu,X)&=\la D_X \nu, N\ra\\
               &=\la D_X (\cos \theta \bn+\sin \theta \b{N}), -\sin \theta \bn+\cos\theta \b{N}\ra\\
               &=\la D_X \bn, \b{N}\ra\\
               &=-II(\bn,X)\\
               &=0,
\end{align*}
where we used $\theta$ is constant, $\bn$ and $\b{N}$ are unit vectors, and $\pt B$ is totally umbilical. Thus we complete the proof of Lemma \ref{lem3}.
\end{proof}

\begin{lem}\label{lem1}
Along $\pt M$, there holds
\begin{equation}
\phi_\nu-q\phi=0.
\end{equation}
\end{lem}
\begin{proof}
First from \eqref{eq7} and Lemma \ref{lem3} we have
\begin{align*}
\phi_\nu &=\langle \xi, -(1+|x|^2)N+2\langle x,N\rangle x\rangle_\nu\\
         &=\la \xi,-2\la x,\nu\ra N+(1+|x|^2)\sigma(\nu,\nu)\nu-2\la x,\sigma(\nu,\nu)\nu\ra x+2\la x,N\ra \nu\ra\\
         &=2\la \xi,-\la x,\nu\ra N+\sigma(\nu,\nu)(\nu-\la x,\nu\ra x)+\la x,N\ra \nu\ra,
\end{align*}
where in the third line we used $|x|=1$ along $\pt M$.

Next noticing that $x=\b{N}=\cos \theta N+\sin\theta \nu$, we get
\begin{align*}
\phi_\nu&=2\la \xi,-\sin\theta N+\sigma(\nu,\nu)(\nu-\sin \theta (\cos \theta N+\sin\theta \nu))+\cos\theta \nu\ra\\
        &=2\la \xi,(\sigma(\nu,\nu)\cos\theta+1)(\cos\theta \nu-\sin\theta N)\ra.
\end{align*}
On the other hand, there holds
\begin{align*}
q\phi&=(\f{1}{\sin\theta}+\cot \theta \sigma(\nu,\nu))\la \xi,-(1+|x|^2)N+2\la x,N\ra x\ra\\
     &=(\f{1}{\sin\theta}+\cot \theta \sigma(\nu,\nu))2\la \xi,-N+\cos \theta (\cos \theta N+\sin\theta \nu)\ra\\
     &=(\f{1}{\sin\theta}+\cot \theta \sigma(\nu,\nu))2\la \xi,-\sin^2\theta N+\cos \theta \sin\theta \nu \ra\\
     &=(1+\cos \theta \sigma(\nu,\nu))2\la \xi,-\sin \theta N+\cos \theta \nu \ra,
\end{align*}
where again in the second line we used $|x|=1$ along $\pt M$.

Hence we obtain
\begin{equation*}
\phi_\nu-q\phi=0.
\end{equation*}

\end{proof}

The next lemma, which indicates the geometric meaning of Lemma \ref{lem1}, may have its own interest. Thus we also include it here.

\begin{lem}
Under the flow $f_t$, there holds
\begin{equation}
\f{d}{dt}\bigg|_{t=0} \theta(t)=\phi_\nu-q\phi.
\end{equation}
In particular, since $f_t$ is conformal (angle preserving), $\phi_\nu-q\phi=0$.
\end{lem}
\begin{proof}
Following \cite{RS}, we denote by a ``prime'' the convariant derivative $\f{D}{dt}\big|_{t=0}$. Also by the appendix of \cite{RS}, we have
\begin{align*}
\nu'&=(\f{\pt \phi}{\pt \nu}+\sigma(Y_0,\nu))N+\phi S_0(\nu)-\phi\sigma(\nu,\nu)\nu-S_1(Y_1)+\cot \theta \td{\nabla}\phi,\\
\bn'&=-\Pi(Y,\bn)\b{N}-S_2(Y_1)+\f{1}{\sin \theta}\td{\nabla \phi},
\end{align*}
where $\td{\nabla}$ denotes the gradient on $\pt M$, $Y_0$ (resp. $Y_1$) the tangent part of the variational vector field $Y$ of $M$ (resp. to $\pt M$), $S_0$ the shape operator of $M$ in $\R^{n+1}$ with respect to $N$, and $S_1$ (resp. $S_2$) the shape operator of $\pt M$ in $M$ (resp. $\pt B$) with respect to $\nu$ (resp. $\bn$).

Note that $\cos \theta(t)=\la \nu,\bn\ra$, which implies
\begin{equation*}
-\sin \theta \f{d}{dt}\bigg|_{t=0} \theta(t)=\la \nu',\bn\ra+\la \nu,\bn'\ra.
\end{equation*}
Taking into account that
\begin{align*}
\bn&=-\sin \theta N+\cos \theta \nu,\\
\b{N}&=\cos\theta N+\sin \theta \nu,
\end{align*}
we have
\begin{align*}
-\sin \theta \f{d}{dt}\bigg|_{t=0} \theta(t) & =\la (\f{\pt \phi}{\pt \nu}+\sigma(Y_0,\nu))N+\phi S_0(\nu)-\phi\sigma(\nu,\nu)\nu,-\sin \theta N+\cos \theta \nu\ra\\
                                          &+\la \nu,-\Pi(Y,\bn)(\cos\theta N+\sin \theta \nu)\ra \\
                                          &= -\sin \theta(\f{\pt \phi}{\pt \nu}+\sigma(Y_0,\nu))-\sin\theta\Pi(Y,\bn),
\end{align*}
or
\begin{equation}\label{eq1}
\f{d}{dt}\bigg|_{t=0} \theta(t) = \f{\pt \phi}{\pt \nu}+\sigma(Y_0,\nu)+\Pi(Y,\bn).
\end{equation}
Again from the appendix of \cite{RS}, there hold
\begin{gather*}
Y_0=Y_1-\cot \theta \phi \nu,\nonumber\\
Y=Y_1-\f{1}{\sin \theta}\phi\bn,\nonumber\\
\sigma(Y_1,\nu)+\Pi(Y_1,\bn)=0.
\end{gather*}

Now plugging these equalities into \eqref{eq1}, the lemma follows immediately.

\end{proof}

\begin{lem}
\begin{equation}
L\phi=-2n\langle \xi,N+Hx\rangle.
\end{equation}
\end{lem}

\begin{proof}
The proof is a direct calculation using moving frame method. This method is very powerful in differential geometry. Take an orthonormal basis $\{e_i,i=1,\cdots,n;e_{n+1}=N \}$. Then we have the structure equations: (see e.g. \cite{CL})
\begin{gather*}
dx=\sum_{i=1}^n \omega_i e_i,\\
de_i=\sum_{j=1}^n \omega_{ij}e_j+\sum_{j=1}^n h_{ij}\omega_j e_{n+1},\\
de_{n+1}=-\sum_{i,j=1}^n h_{ij} \omega_ie_j,
\end{gather*}
where $\omega_i$ is the dual forms and $\omega_{ij}$ the connection forms. Thus there holds
\begin{align}
\Delta \phi&=\Delta \langle \xi, -(1+|x|^2)N+2\langle x,N\rangle x \rangle \nonumber\\
           &=\la \xi,-\Delta((1+|x|^2)N)+2\Delta (\la x,N\ra x)\ra \nonumber\\
           &=\la \xi,-(\Delta |x|^2 \cdot N+2 \sum_{i=1}^n(|x|^2)_{,i} N_{,i}+(1+|x|^2)\Delta N) \nonumber\\
           &+2(\Delta \la x,N\ra \cdot x+2\sum_{i=1}^n\la x,N\ra_{,i}x_{,i}+\la x,N\ra \Delta x)\ra.\label{eq6}
\end{align}
Note that
\begin{align*}
\Delta |x|^2&=2\la x,\Delta x\ra+2|\nabla x|^2\\
            &=2nH\la x,N\ra+2n,\\
\sum_{i=1}^n(|x|^2)_{,i}N_{,i}&=-2\sum_{i,j=1}^n\la x,e_i\ra h_{ij}e_j.
\end{align*}
And using Codazzi equation $\sum_{i=1}^nh_{ij,i}=\sum_{i=1}^n h_{ii,j}=nH_{,j}=0$ we have
\begin{align*}
\Delta N&=\sum_{i=1}^n N_{,ii}=\sum_{i,j=1}^n (-h_{ij}e_j)_{,i}\\
        &=-\sum_{i,j=1}^n h_{ij}h_{ij}N=-|\sigma|^2N.
\end{align*}
Moreover we can get
\begin{align*}
\Delta \la x,N\ra&=\la \Delta x,N\ra+2\sum_{i=1}^n \la x_{,i},N_{,i}\ra +\la x,\Delta N\ra\\
                 &=\la nHN,N\ra+2\sum_{i,j=1}^n\la e_i,-h_{ij}e_j\ra+\la x,-|\sigma|^2 N\ra\\
                 &=-nH-|\sigma|^2\la x,N\ra,\\
\sum_{i=1}^n \la x,N\ra_{,i}x_{,i}&=\sum_{i,j=1}^n\la x,-h_{ij}e_j\ra e_i=\sum_{i,j=1}^n-h_{ij}\la x,e_j\ra e_i.
\end{align*}
Now substituting all these terms into \eqref{eq6} gives rise to
\begin{align*}
\Delta \phi&=\la \xi,-((2nH\la x,N\ra+2n) \cdot N-4\sum_{i,j=1}^n\la x,e_i\ra h_{ij}e_j-(1+|x|^2)|\sigma|^2N) \\
           &+2((-nH-|\sigma|^2\la x,N\ra) \cdot x-2\sum_{i,j=1}^n h_{ij}\la x,e_j\ra e_i+\la x,N\ra nHN)\ra\\
           &=\la \xi,(-2n+(1+|x|^2)|\sigma|^2)N-2(nH+|\sigma|^2\la x,N\ra)x\ra\\
           &=\la \xi, -2n(N+Hx)\ra-|\sigma|^2\phi.
\end{align*}

Therefore,
\begin{equation*}
L\phi=\Delta \phi+|\sigma|^2\phi=-2n\langle \xi,N+Hx\rangle.
\end{equation*}

\end{proof}

Thus we obtain
\begin{equation*}
\pt^2 E(\phi)=-2n\int_M \langle \xi,N+Hx\rangle\cdot \langle \xi, (1+|x|^2)N-2\langle x,N\rangle x\rangle da.
\end{equation*}

To analyze $\pt^2 E(\phi)$, we define a quadratic form
\begin{equation}
Q(\xi_1,\xi_2)=-2n\int_M \langle \xi_1,N+Hx\rangle\cdot \langle \xi_2, (1+|x|^2)N-2\langle x,N\rangle x \rangle da,
\end{equation}
for $\forall\, \xi_1,\xi_2\in \SS^n$. Denote by $\{\pt_A\}_{A=1}^{n+1}$ the standard coordinate vectors in $\R^{n+1}$. Then we have the following lemma.
\begin{lem}\label{lem2}
$Q$ has the following properties.
\begin{enumerate}
  \item $Q$ is symmetric.
  \item $tr Q=\sum\limits_{A=1}^{n+1} Q(\pt_A,\pt_A)\leq 0$ with equality if and only if $|x|=const$ on M.
\end{enumerate}
\end{lem}
\begin{proof}
(1) First we prove $Q$ is symmetric. Note that in fact $Q$ is defined as
\begin{equation*}
Q(\xi_1,\xi_2)=-\int_M L(\phi[\xi_1])\cdot \phi[\xi_2] da.
\end{equation*}
Then Green's formula implies
\begin{align*}
Q(\xi_1,\xi_2)&=-\int_M \phi[\xi_1]\cdot L(\phi[\xi_2]) da+\int_{\pt M} \left(\phi[\xi_1](\phi[\xi_2])_\nu-(\phi[\xi_1])_\nu\phi[\xi_2]\right)ds.\\
\end{align*}
But Lemma \ref{lem1} yields $(\phi[\xi_i])_\nu=q\phi[\xi_i]$, $i=1,2$. So the boundary term vanishes and then
\begin{equation*}
Q(\xi_1,\xi_2)=Q(\xi_2,\xi_1).
\end{equation*}

(2) Next we calculate $tr Q$.
\begin{align*}
tr Q&=\sum\limits_{A=1}^{n+1} Q(\pt_A,\pt_A)\\
    &=-2n\int_M \sum_{A=1}^{n+1} \langle \pt_A,N+Hx\rangle\cdot \langle \pt_A, (1+|x|^2)N-2\langle x,N\rangle x\rangle da\\
    &=-2n\int_M \la N+Hx,(1+|x|^2)N-2\langle x,N\rangle x\ra da\\
    &=-2n\int_M \left(H\la x,N\ra (1-|x|^2)+1+|x|^2-2\la x,N\ra^2\right)da\\
    &\leq -2n\int_M (H\la x,N\ra +1)(1-|x|^2)da.
\end{align*}
Also we have $\Delta |x|^2=2n(H\la x,N\ra +1)$. Consequently,
\begin{align*}
tr Q&\leq -\int_M \Delta |x|^2\cdot (1-|x|^2)da\\
    &=\int_M \nabla |x|^2 \cdot \nabla (1-|x|^2)da-\int_{\pt M} \f{\pt |x|^2}{\pt \nu}(1-|x|^2)ds\\
    &=-\int_M |\nabla (|x|^2)|^2da\\
    &\leq 0,
\end{align*}
where we have used $|x|=1$ on $\pt M$ to remove the boundary term. And it is easy to see $tr Q=0$ if and only if $|x|=const$.

So we complete the proof.
\end{proof}

Thus $Q$ has at least one negative eigenvalue. But on the other hand,
\begin{equation}
div_{\R^{n+1}}Y[\xi]=\sum_{A=1}^{n+1}\la D_{\pt_A}Y[\xi],\pt_A\ra=2(n+1)\langle \xi,x\rangle,
\end{equation}
which by integration implies
\begin{align}
\int_M \phi da&=\int_M\langle Y[\xi],N\rangle da\nonumber\\
              &=-\int_T div_{\R^{n+1}}Y[\xi] dv+\int_\Omega \langle Y[\xi],\b{N} \rangle da\nonumber\\
              &=-2(n+1)\int_T \langle \xi,x\rangle dv.
\end{align}
So generally $\int_M \phi da \neq 0$. That means $\phi[\xi]$ is not a test function.

However, under the hypothesis of Theorem \ref{thm1} that the mass center of $T$ is at the origin, we have $\int_M \phi da=-2(n+1)\int_T \langle \xi,x\rangle dv=0$ for $\forall\, \xi \in\SS^n$. So if we choose $\xi$ as an eigenvector corresponding to the negative eigenvalue of $Q$, we have $\pt^2 E(\phi[\xi])=Q(\xi,\xi)<0$, which implies that $M$ is unstable. This completes the proof of Theorem \ref{thm1}.

\section{Other applications and question}\label{sec4}

In this section we shall give several applications of the above argument and propose a conjecture on the topic.

\subsection{Another criteria for instability}

The following proposition is an immediate result.
\begin{prop}
If the quadratic form $Q$ has two negative eigenvalues, then $M$ is unstable.
\end{prop}
\begin{proof}
Assume $Q$ is diagonalized such that $\xi_1$ and $\xi_2$ are the eigenvectors corresponding to the two negative eigenvalues. Then for real numbers $c_1$ and $c_2$ with $c_1^2+c_2^2\neq 0$,
\begin{equation}\label{eq5}
Q(c_1\xi_1+c_2\xi_2,c_1\xi_1+c_2\xi_2)=c_1^2Q(\xi_1,\xi_1)+c_2^2Q(\xi_2,\xi_2)<0.
\end{equation}
On the other hand,
\begin{align*}
\int_M \phi[c_1\xi_1+c_2\xi_2] da&=-2(n+1)\int_T \langle c_1\xi_1+c_2\xi_2,x\rangle dv\\
                                 &=-2(n+1)(c_1\int_T \langle \xi_1,x\rangle dv+c_2\int_T \langle \xi_2,x\rangle dv).
\end{align*}
So we can always choose proper $c_1$ and $c_2$ with $c_1^2+c_2^2\neq 0$ such that $\int_M \phi[c_1\xi_1+c_2\xi_2] da=0$. Then using $\phi[c_1\xi_1+c_2\xi_2]$ as a test function, from \eqref{eq5} we know $M$ is unstable.

\end{proof}
The significance of the above proposition is as follows. For a given concrete capillary hypersurface $M$ in $B^{n+1}$, the quadratic form $Q$ is computable in principle. Then if $Q$ has two negative eigenvalues, we can assert its instability. Also from this proposition we know that for hyperplanes and spherical caps $Q$ has exactly one negative eigenvalue.

\subsection{The mass center of minimal submanifolds with free boundary}

By free boundary we mean that $M$ intersects $\pt B^{n+1}$ orthogonally, that is, $\nu=x$ along $\pt M$. By analyzing the vector field $Y[\xi]$, we have the following proposition.
\begin{prop}\label{prop1}
The mass center of a minimal submanifold $M^k$ with free boundary in a Euclidean ball is at the origin.
\end{prop}

\begin{proof}
Along $M^k$ choose the orthonormal basis $\{e_i,i=1,\cdots,k;e_\alpha, \alpha=k+1,\cdots,n+1\}$ such that $\{e_i,i=1,\cdots, k\}\subset TM$. Then we have
\begin{align*}
div_M Y[\xi]^T=div_M(Y[\xi]-\sum_\alpha \la Y[\xi],e_\alpha \ra e_\alpha)=2k\la\xi,x\ra+\la Y[\xi], k\vec{H}\ra=2k\la\xi,x\ra.
\end{align*}
By divergence theorem, we have
\begin{equation*}
2k\int_M\la\xi,x\ra da=\int_{\pt M} \la Y[\xi]^T,\nu\ra ds=\int_{\pt M} \la Y[\xi],x\ra ds=0,
\end{equation*}
where we have used the fact $Y[\xi]$ is tangential to $\pt B^{n+1}$.
\end{proof}

This proposition shows that minimal submanifolds with free boundary have some symmetry. Comparing with it, we mention two other properties of $M^k$:
\begin{enumerate}
  \item The mass center of the boundary $\pt M$ is at the origin (a simple argument).
  \item The volume of $M$ has a lower bound $|M^k|\geq |B^k|$, where $B^k$ is a $k$-dimensional unit ball (\cite{B,FS,RV}).
\end{enumerate}

\subsection{Stable immersed closed CMC hypersurfaces in $\R^{n+1}$}
At last we give a new proof of a theorem by Barbosa and do Carmo.
\begin{thm}[\cite{BdC}]
The only stable immersed closed hypersurface of constant mean curvature in $\R^{n+1}$ is the round sphere.
\end{thm}
\begin{proof}
By translation, assume the mass center of generalized body $T$ enclosed by $M$ is at the origin. So $\int_M \phi[\xi]da=0$ for all $\xi\in \SS^n$. If $M$ is the round sphere, we are done. Otherwise $|x|\neq const$. So by Lemma \ref{lem2} the quadratic form $Q$ has a negative eigenvalue. Choosing $\xi$ as an eigenvector corresponding to the negative eigenvalue, we have
\begin{equation*}
\pt^2 E(\phi[\xi])=-\int_M L\phi[\xi]\cdot \phi[\xi]da=Q(\xi,\xi)<0,
\end{equation*}
which shows that $M$ is unstable.
\end{proof}

\subsection{An open question}

Since all the examples, i.e. the Delaunay type capillary hypersurfaces is known to be stable or unstable, we propose a conjecture as follows.
\begin{conj}
The only stable capillary hypersurface $M^n\,(n\geq 3)$ in a unit Euclidean ball $B^{n+1}$ is a totally geodesic hyperplane or a spherical cap.
\end{conj}

There are some remarks on this conjecture.
\begin{enumerate}
  \item For $n\geq 2$, $H=0$ and $\theta=\f{\pi}{2}$, $M$ must be totally geodesic \cite{RV}.
  \item For $n=2$ and $\theta=\f{\pi}{2}$, $M$ is a totally geodesic disk, a spherical cap or a surface of genus $1$ with embedded boundary having at most two connected components \cite{RV}.
  \item For $n=2$ and $H=0$, $M$ is a totally geodesic disk or a surface of genus $1$ with at most three connected boundary components \cite{RS}.
\end{enumerate}


\bibliographystyle{Plain}

\end{document}